\pdfoutput=1
\RequirePackage{ifpdf}
\ifpdf 
\documentclass[pdftex]{sigma}
\else
\documentclass{sigma}
\fi

\numberwithin{equation}{section}
\newtheorem{thm}{Theorem}[section]

{\theoremstyle{definition}
\newtheorem{rem}[thm]{Remark}

\newtheorem{pc}[thm]{Particular case}}

\usepackage{tikz}

\begin{document}
\allowdisplaybreaks

\renewcommand{\thefootnote}{}

\newcommand{\arXivNumber}{2306.11110}

\renewcommand{\PaperNumber}{079}

\FirstPageHeading

\ShortArticleName{About a Family of ALF Instantons with Conical Singularities}

\ArticleName{About a Family of ALF Instantons\\ with Conical Singularities\footnote{This paper is a~contribution to the Special Issue on Differential Geometry Inspired by Mathematical Physics in honor of Jean-Pierre Bourguignon for his 75th birthday. The~full collection is available at \href{https://www.emis.de/journals/SIGMA/Bourguignon.html}{https://www.emis.de/journals/SIGMA/Bourguignon.html}}}

\Author{Olivier BIQUARD~$^{\rm a}$ and Paul GAUDUCHON~$^{\rm b}$}

\AuthorNameForHeading{O.~Biquard and P.~Gauduchon}

\Address{$^{\rm a)}$~Sorbonne Universit\'e and Universit\'e Paris Cit\'e, CNRS, IMJ-PRG, F-75005 Paris, France}
\EmailD{\href{mailto:olivier.biquard@sorbonne-universite.fr}{olivier.biquard@sorbonne-universite.fr}}

\Address{$^{\rm b)}$~\'Ecole Polytechnique, CNRS, CMLS, F-91120 Palaiseau, France}
\EmailD{\href{mailto:paul.gauduchon@polytechnique.edu}{paul.gauduchon@polytechnique.edu}}

\ArticleDates{Received June 21, 2023, in final form October 10, 2023; Published online October 20, 2023}

\Abstract{We apply the techniques developed in our previous article to describe some interesting families of ALF gravitational instantons with conical singularities. In particular, we completely understand the 5-dimensional family of Chen--Teo metrics and prove that only 4-dimensional subfamilies can be smoothly compactified so that the metric has conical singularities.}


\Classification{53C25; 53C55}

\begin{flushright}
\begin{minipage}{65mm}
 \emph{Dedicated to Jean-Pierre Bourguignon on the occasion of his 75th birthday,\\
 with our admiration and gratitude.}
\end{minipage}
\end{flushright}

\renewcommand{\thefootnote}{\arabic{footnote}}
\setcounter{footnote}{0}

\section{Introduction} In a previous paper \cite{BG}, the authors of the present paper have provided a complete classification, as well as an effective mode of construction, of so-called {\it toric Hermitian ALF gravitational instantons}. These are four-dimensional, complete, non-compact oriented Ricci-flat Riemannian (positive definite, smooth) manifolds, which are toric, i.e., admits an effective metric action of the torus $\mathbb{T} ^2 = S^1 \times S^1$, are conformally K\"ahler -- but non-K\"ahler -- and, at infinity, are diffeomorphic to the product $\mathbb{R} \times L$, where $L$ is locally a $S^1$-bundle over the sphere $S^2$; the AF case is when $L = S ^2 \times S^1$.

This class of gravitational instantons includes the Riemannian versions, obtained by {\it Wick rotations}, of well-known Lorentzian space-times, namely (i) the {\it Schwarzschild space}, (ii) the family of {\it Kerr spaces}, (iii) the {\it self-dual Taub-NUT space}, equipped with the opposite orientation to the one induced by its hyperk\"ahler structure, and (iv) the {\it Taub-bolt space}, discovered in~1978 by Don Page \cite{P}. Apart from the Taub-NUT space, these spaces share the feature of being {\it of type~$D^+D^-$}, meaning that their self-dual and anti-self-dual Weyl tensors, $W^+$ and~$W^-$ respectively, are both {\it degenerate} and non-vanishing, hence giving rise to an {\it ambik\"ahler structure}, as defined in \cite{ACG}, in their conformal cases, cf.\ Section~\ref{sskaehler} below, see also \cite{dixon} for the Kerr spaces (in this terminology, the self-dual Taub-NUT space is of type $D^+O^-$, meaning that $W ^+$ is degenerate and non-vanishing, while $W ^- \equiv 0$).

It has been conjectured for a long time that no other $4$-manifold could be admitted in the family of toric Hermitian ALF gravitational instantons.
In 2011 however, a new 1-parameter family of toric ALF -- in fact AF -- gravitational instantons was discovered by Yu Chen and Edward Teo, \cite{CT2}, and these instantons were shown to be Hermitian by Steffen Aksteiner and Lars Andersson \cite{AA}. In contrast with the former examples, the {\it Chen--Teo instantons} are {\it not} derived from a Lorentzian space by a Wick rotation, and don't give rise anymore to an ambitoric structure, but nevertheless still admit a toric K\"ahler structure in their conformal class. Moreover, in view of the ALF condition, this toric K\"ahler structure, as well as the toric K\"ahler structure associated to any toric Hermitian ALF gravitational instanton, can be chosen to be locally close at infinity to the product of the standard sphere $S^2$, of curvature $1$, with a hyperbolic cusp, of curvature $-1$.

In \cite{BG}, it has been shown that, together with the above mentioned gravitational instantons, issuing from classical Lorentzian spaces, the Chen--Teo instanton provides the missing puzzle-piece needed for a full classification of the toric Hermitian ALF gravitational instantons. In this approach, the conformal K\"ahler structure plays a crucial role and, together with a recent ansatz due to Paul Tod, allows for a simple construction of these instantons, as explained in the next section. In particular, it provides a new, simple definition of the Chen--Teo instantons. As a~matter of fact, this approach first provides a general description of these metrics on the open set where the toric action is free. The actual construction of instantons then amounts to smoothly compactifying the manifold together with the metric along the invariant divisors, encoded by the edges of the moment polytope. This phase of the construction actually requires strong restrictions, and eventually ends up with the above mentioned complete classification.

A larger class is obtained by smoothly compactifying the manifold as above but by allowing {\it conical singularities} of the metric along some of them. Metrics of this kind will be called {\it regular}. Examples of such metrics can be found inside the well-known family of so-called {\it Kerr--Taub-NUT} metrics introduced by Gibbons--Perry in \cite{GP}, cf.\ \cite[Section~7.2.1]{BG}.

In 2015, Chen--Teo introduced a new 5-parameter family of toric ALF Ricci-flat metrics -- in fact, a 4-parameter family, if we ignore the scaling, as we shall do in the sequel -- including the above mentioned 1-parameter sub-family of Chen--Teo instantons.

Apart from the latter, none of them is smooth -- a consequence of the classification in \cite{BG}, see also Section~\ref{ss-smooth} below -- but some of them are regular, as defined above.
The main goal of this paper is to provide an alternative description of the Chen--Teo family, to detect inside this family the regular elements and the topology of their boundary at infinity, cf.\ Theorem \ref{thm-regular} below, and to describe some distinguished regular sub-families and their limits.

\section{Toric Hermitian ALF gravitational instantons: a quick review} In this section, we recall some main features of toric Hermitian ALF gravitational instantons, taken from~\cite{BG}.

\subsection{General presentation} It turns out, cf.\ \cite[Corollary 5.2]{BG}, that each of them is completely encoded by a positive, continuous, convex, piecewise affine function $f$ on $\mathbb{R}$ of the form
\begin{gather} \label{f} f (z) = A + \sum _{i = 1} ^r a_i |z - z_i|, \end{gather}
for some positive integer $r$, where $z$ denotes the standard parameter of $\mathbb{R}$, $A$ is a positive real number, the coefficients $a _i$ are positive real numbers, with $\sum _{i = 1}^r a_i = 1$, and the $z_i$, $i = 1, \dots, r$, called {\it angular} or {\it turning} points, denote the points of discontinuity on $\mathbb{R}$ of the slope of $f$. For convenience, we put: $z _0 := - \infty$ and $z_{r + 1} := + \infty$, cf.\ Figure~\ref{fig:affine} for the piecewise affine function of the Chen--Teo instanton. On each open interval $(z_i, z_{i + 1})$, $i = 0, \dots, r$, the slope of $f$ is constant, denoted by $f' _i$. It is required that
\[
 f'_0 = - 1 < f'_1 < \dots < f _{r-1} < f'_r = 1.
 \]
The coefficients $a_i$ are related to the slopes $f'_i$ by
\[ a _i = \frac{1}{2} (f'_i - f'_{i - 1}), \qquad i = 1, \dots, r.\]

According to the {\it Tod ansatz}, cf.\ \cite{T} and also \cite[Section~3]{BG}, the geometry of a toric Hermitian Ricci-flat metric is determined by a harmonic, axisymmetric (real) function $U = U(\rho, z)$, defined on the Euclidean space $\mathbb{R} ^3$, with the following notation: if $u_1$, $u_2$, $u_3$ denote the standard coordinates of $\mathbb{R} ^3$, the pair $(\rho, z)$ -- the so-called {\it Weyl--Papapetrou coordinates} -- are defined by \smash{$\rho := \big(u_1 ^2 + u _2 ^2\big)^{\frac{1}{2}}$}, $z = u _3$; $U$ being axisymmetric means that it is invariant by the $S^1$-action: ${\rm e}^{{\rm i} \theta} \cdot u = (\cos{\theta} u_1 + \sin{\theta} u_2, - \sin{\theta} u_1 + \cos{\theta} u_2, u_3)$, hence is a function of $\rho$, $z$, and the condition of being harmonic is then expressed by $U_{z z} + U_{\rho \rho} + \frac{1}{\rho} U_{\rho} = 0$, where, as usual, $U_{\rho}$, $U_z$, $U _{\rho, z}$ etc.\ denote the partial derivatives with respect to $\rho$ and $z$.
For any such {\it generating function} $U$, the corresponding metric is then given, in the {\it Harmark form}~\cite{CT1,harmark}, by
\begin{gather} \label{g-harmark} g = \frac{({\rm d}t - F {\rm d} x _3) ^{\otimes 2}}{V} + V \rho ^2 {\rm d} x _3 \otimes {\rm d} x _3 + {\rm e}^{2 \nu} ({\rm d} \rho \otimes {\rm d} \rho + {\rm d} z \otimes {\rm d} z), \end{gather}
on the open set, where $\rho \neq 0$,
where $t$, $x_3$ are angular coordinates, and $V$, $F$, ${\rm e}^{2 \nu}$ are functions of $\rho$, $z$, defined by
\begin{gather} V = - \frac{1}{k} \left(\rho U_{\rho} + \frac{U_{\rho} ^2 U_{z z}}{U_{\rho z} ^2 + U _{z z} ^2}\right), \qquad {\rm e}^{2 \nu} = \frac{1}{4} V \rho ^2 \big(U_{\rho z} ^2 + U_{z z} ^2\big), \label{V-nu}\\
F = - \frac{1}{k} \left(- \frac{\rho U_{\rho} ^2 U_{\rho z}}{U_{\rho z}^2 + U_{z z} ^2} + \rho ^2 U_z + 2 H\right), \label{F} \end{gather}
where $H$, the {\it conjugate function} of $U$, is defined, up to an additive constant, by
\begin{gather} \label{H} H _z = \rho U _{\rho}, \qquad H _{\rho} = - \rho U_z, \end{gather}
cf.\ Section~\ref{sskaehler} for the significance of the constant $k$.

The functions $H$ and $F$ are both defined up to an additional constant. Indeed, in the expression \eqref{g-harmark} of the metric, the $1$-form $\eta = {\rm d}t - F {\rm d} x _3$ is well defined, but the pair $(t, F)$ is subject to the transform $(t, F) \mapsto (t + c x_3, F + c)$, for any constant $c$, by which $\eta$, the vector field $\partial _t$ and $x _3$ remain unchanged, while the vector field $\partial _{x_3}$ becomes $\partial _{x _3} - c \partial _t$. In particular, the vector field $\partial _{x_3} + F \partial _t$ remains unchanged.

In the current ALF case, it was shown in \cite[Section~5]{BG} that the generating function $U$ of any toric Hermitian ALF gravitational instanton is defined on the whole space $\mathbb{R} ^3$, except on the $z$-axis $\rho = 0$, that, near the $z$-axis, $U$ is close to $f (z) \log{\rho ^2}$, while, at infinity, it is asymptotic to the harmonic axisymmetric function $U_0$ defined by
\[
 U _0 (\rho, z) = 2 \big(\rho ^2 + z ^2\big) ^{\frac{1}{2}} - z \log{\frac{\big(\rho ^2 + z^2\big) ^{\frac{1}{2}} + z}{\big(\rho ^2 + z^2\big) ^{\frac{1}{2}}- z}}.\]
It follows that the generating function $U$ of any toric Hermitian ALF gravitational instanton is actually entirely determined by the above piecewise affine function $f (z)$, via the formula
\begin{gather} \label{U} U (\rho, z) = A \log{\rho ^2} + \sum _{i = 1} ^r a _i U_0 (\rho, z - z_i). \end{gather}

By setting
\[ d_i := \big(\rho ^2 + (z - z_i)^2\big)^{\frac{1}{2}}, \]
and by noticing that the constant $k$ in \eqref{V-nu}--\eqref{F} is equal to $2 A$, cf.\ below, we get the following expressions of $U$, its first and second derivatives, and $H$:
\begin{gather} U (\rho, z) = A \log{\rho ^2} + 2 \sum _{i = 1}^r a_i d_i - \sum _{i = 1}^r a _i (z - z_i) \log{\frac{(d_i + z - z_i)}{(d_i - z + z_i)}},\label{U-def} \\
U_{\rho} = \frac{2}{\rho} \bigg(A + \sum _{i = 1}^r a_i d_i\bigg), \qquad U_z = - \sum _{i = 1}^r a_i \log{\frac{(d_i + z - z_i)}{(d_i - z + z_i)}}, \label{U-rho-z} \\
 U_{\rho \rho} = - \frac{2}{\rho ^2} \bigg(A + \sum _{i = 1} ^r a _i d _i\bigg) + 2 \sum _{i = 1} ^r \frac{a_i}{d_i}, \qquad U_{\rho z} = \frac{2}{\rho} \sum _{i = 1} ^r a_i \frac{(z - z_i)}{d_i}, \qquad
U _{z z} = - 2 \sum _{i = 1} ^r \frac{a _i}{d_i},\nonumber
 \end{gather}
and
\[
H (\rho, z) = 2 A z + \sum _{i = 1}^r a _i (z - z_i) d _i + \frac{1}{2} \rho ^2 \sum_{i = 1}^r a _i \log{\frac{(d_i + z - z_i)}{(d_i - z + z_i)}}, \]
up to constant.
We then get
\begin{gather}
 V = \frac{1}{A} \biggl(A + \sum _{i = 1} ^r a_i d_i\biggr) \left(\frac{\bigl(\sum _{i = 1}^r \frac{a_i}{d_i}\bigr) \bigl(A + \sum _{i = 1} ^r a_i d_i\bigr)}{\bigl(\sum _{i = 1}^r \frac{a _i (z - z_i)}{d_i}\bigr) ^2 + \bigl(\sum _{i = 1}^r \frac{a _i \rho}{d_i}\bigr)^2} - 1\right), \label{V-def} \\
 {\rm e}^{2 \nu} = \frac{1}{A} \biggl(A + \sum _{i = 1} ^r a_i d_i\biggr) \biggl(\sum _{i = 1}^r \frac{a_i}{d_i} \biggl(A + \sum _{i = 1} ^r a_i d_i\biggr) - \biggl(\biggl(\sum _{i = 1}^r \frac{a _i (z - z_i)}{d_i}\biggr) ^2 \nonumber\\
 \hphantom{{\rm e}^{2 \nu} =}{}+ \biggl(\sum _{i = 1}^r \frac{a _i \rho}{d_i}\biggr)^2\biggr)\biggr), \label{nu-def}
 \end{gather}
and
\begin{gather} \label{F-def} F = \frac{1}{A} \left(\frac{\bigl(A + \sum _{i = 1} ^r a_i d_i\bigr)^2 \bigl(\sum _{i = 1}^r \frac{a _i (z - z_i)}{d_i}\bigr)}{\bigl(\sum _{i = 1}^r \frac{a _i (z - z_i)}{d_i}\bigr) ^2 + \bigl(\sum _{i = 1}^r \frac{a _i \rho}{d_i}\bigr)^2} - 2 A z - \sum _{i = 1}^r a_i (z - z_i) d_i\right). \end{gather}

It is easy to show that
\[\bigg(\sum _{i = 1} ^r \frac{a _i (z - z_i)}{d_i}\bigg)^2 + \bigg(\sum _{i = 1} ^r \frac{a _i \rho}{d_i}\bigg)^2 \leq 1,
\qquad \sum _{i = 1} ^r \frac{a_i}{d_i} \sum _{i = 1} ^r a _i d _i \geq 1.
\]
It then readily follows that
\[ V \geq 1 + A \sum _{i = 1} ^r \frac{a_i}{d_i}, \]
and that $V$ tends to $1$ at infinity.

On the $z$-axis $\rho = 0$, for any $z$ where $f'(z) \neq 0$, we infer
\begin{gather} V (0, z) = \frac{1}{A} \frac{f (z)}{(f' (z))^2} \left(f (z) \sum _{i = 1}^r \frac{a_i}{|z - z _i|} - (f' (z))^2\right),\label{V-0}\\
{\rm e}^{2 \nu} (0, z) = (f' (z)) ^2 V(0, z), \nonumber\\ 
F (0, z) = \frac{1}{A} \left(\frac{(f (z))^2}{f' (z)} - 2 A z - \sum _{i = 1}^r a _i (z - z_i) |z - z _i|\right) = \frac{1}{A} \left(\frac{(f (z))^2}{f' (z)} - H (0, z)\right). \label{F-0}
 \end{gather}
From \eqref{F-0}, we infer that on any interval $(z_i, z_{i + 1})$, $i = 0, \dots, r$, {\it $F (0, z)$ is constant}, say equal to $F _i$. If $f' _i \neq 0$ and $f' _{i - 1} \neq 0$, since $H$ is continuous on the $z$-axis, we then have
\[
 F_i - F_{i - 1} = \frac{1}{A} f_i ^2 \left(\frac{1}{f'_i} - \frac{1}{f'_{i - 1}}\right); \]
if, however, $f'_i = 0$, then $f'_{i - 1} \neq 0$ and $f'_{i + 1} \neq 0$ and we then get
\[
F_{i + 1} - F_{i - 1} = \frac{1}{A} \left(f_i ^2 \left(\frac{1}{f'_{i + 1}} - \frac{1}{f'_{i - 1}}\right) - 2 (z _{i + 1} - z _i) f_i\right), \]
cf.\ \cite[Propsition 7]{BG}.
From \eqref{F-0} again, we get
\begin{gather*}F _0 = - \frac{1}{A} \bigg(A + \sum _{i = 1} ^r a _i z _i\bigg) ^2 + \frac{1}{A } \sum _{i = 1} ^r a _i z _i ^2, \\
F _r = \frac{1}{A} \bigg(A - \sum _{i = 1} ^r a _i z _i\bigg) ^2 - \frac{1}{A } \sum _{i = 1} ^r a _i z _i ^2, \end{gather*}
up to an additional constant, hence
\begin{gather} \label{Fr-F0} F_r - F_0 = \frac{2}{A} \bigg(A ^ 2 + \bigg(\sum _{i = 1} ^r a _i z _i\bigg) ^2 - \sum _{i = 1} ^r a_i z _i ^2\bigg). \end{gather}
It follows that
\[ A = \frac{1}{4} \bigg(F_r - F_0 + \bigg((F_r - F_0) ^{\frac{1}{2}} + 16 \sum _{i = 1} ^r a _i z _i ^2 - 16 \bigg(\sum _{i = 1} ^r a _i z _i\bigg) ^2\bigg) ^{\frac{1}{2}}\bigg). \]
In particular, the metric is AF, i.e., satisfies $F_r - F _0 = 0$, if and only if
\[
 A ^2 = \sum _{i = 1}^r a _i z _i ^2 - \bigg(\sum _{i = 1} ^r a _i z _i\bigg) ^2. \]

\begin{rem} \label{rem-change} As explained above, to any convex, piecewise affine function $f (z)$ as defined in~\eqref{f} is associated a generating function $U (\rho, z)$, defined by \eqref{U}, hence by \eqref{U-def}; conversely, it follows from \eqref{U-rho-z} that $f (z)$ is determined by $U (\rho, z)$ via the formula $f (z) = \frac{1}{2} (\rho U_{\rho}) (0, z)$. The corresponding Ricci-flat metric $g$ is then expressed by \eqref{g-harmark}, where the functions $V (\rho, z)$, $F (\rho, z)$, ${\rm e}^{2 \nu} (\rho, z)$ are given by \eqref{V-nu}--\eqref{F}, hence by \eqref{V-def}--\eqref{F-def}. For any real constants $\alpha > 0$,~$\beta$, $f (z)$ may be replaced by the function $\tilde{f} (z) := \frac{1}{\alpha} f (\alpha z + \beta) = \frac{A}{a} + \sum _{i = 1}^r a _i |z - \tilde{z}_i|$, with~$\tilde{z}_i = \frac{z_i - \beta}{\alpha}$, and the corresponding generating function is then replaced by
\[\tilde{U} (\rho, z) = \frac{1}{a} U (\alpha \rho, \alpha z + \beta) = \frac{A}{a} \log{\rho ^2} + \sum _{i = 1} ^r a _i U_0 (\rho, z - \tilde{z} _i).\]
 The corresponding Ricci-flat metric is then $\tilde{g} (\rho, z, t, x_3) := \frac{1}{\alpha ^2} g (\alpha \rho, \alpha z + \beta, \alpha t, x _3)$, meaning that~$\tilde{g}$ is homothetic to $g$ by a factor $1/\alpha ^2$, via the change of variables $(\rho, z, t, x _3) \mapsto (\alpha \rho, \alpha z + \beta, \alpha t, x_3)$. Also notice that, by denoting $\tilde{f} _i := \tilde{f} (\tilde{z} _i)$ and $f_i : = f (z _i)$, we have
$\tilde{f} _i = f_i/\alpha$, $i = 1, \dots, r$. \end{rem}

\subsection{The K\"ahler environment} \label{sskaehler}
By definition, a toric Hermitian ALF gravitational instanton, say $(M, g)$, admits a K\"ahler metric,~$g _K$, in the conformal class of $g$, which is actually toric as well, meaning that the torus action is Hamiltonian, i.e., admits a {\it moment map} The fact that $g$ is conformally K\"ahler implies that the self-dual Weyl tensor $W ^+$ of $g$, regarded as a (symmetric, trace-less) operator on the self-dual part of $\Lambda ^2 M$, is {\it degenerate}, meaning that $W^+$ has a simple, non-vanishing, simple eigenvalue, say $\lambda$, and a repeated eigenvalue $- \frac{\lambda}{2}$. It is also required that $\lambda$ is not constant and everywhere non-vanishing. According to \cite{derdzinski}, the K\"ahler metric $g _K$ is then equal to $\lambda ^{2/3} g$ or a constant multiple. In view of the ALF condition, it is convenient to set: \smash{$g _K = \big(k ^{-1} \lambda\big) ^{2/3} g$,} where the constant $k$ is chosen in such a way that
$g_K$ is asymptotic at infinity to the product of the standard sphere of radius $1$ and the Poincar\'e cusp of sectional curvature $-1$ (more detail in \cite[Section~2]{BG}).

The conformal factor \smash{$\big(k ^{-1} \lambda\big) ^{\frac{2}{3}}$} is then equal to $x _1 ^2$, where $x _1$ denotes the moment of the Hamiltonian Killing vector field $\partial _t$, and
the scalar curvature, ${\rm Scal} _{g_K}$, of the K\"ahler metric~$g _K$~is then equal to \smash{$6 k ^{\frac{2}{3}} \lambda ^{\frac{1}{3}}= 6 k x _1$}, and tends to $0$ at infinity. In particular, $g_K = x _1 ^2 g$ is {\it extremal}, even {\it Bach-flat}, since it is conformal to an Einstein metric. The constant $k$ is actually the same as the constant $k$ appearing in \eqref{V-nu}--\eqref{F}, and turns out to be equal to $2 A$ \cite[equation (89)]{BG}.

In terms of the generating function $U$, the K\"ahler form, $\omega _K$, and the volume form, $v _{g _K}$, of~$g_K$ have the following expression:
\[ \label{omega-K } \omega _K = \frac{2}{U_{\rho} ^2} \Big(\frac{1}{\rho} (U_{z z} {\rm d}\rho - U_{\rho z} {\rm d}z) \wedge ({\rm d}t - F {\rm d} x_3) - V (U_{\rho z} {\rm d} \rho + U_{z z} {\rm d}z) \wedge {\rm d} x _3\Big), \]
and
\begin{gather} \label{vgK} v _{g _K} = \frac{1}{2} \omega _K \wedge \omega _K = \frac{4}{\rho U_{\rho} ^4} \big(U_{\rho z} ^2 + U_{z z} ^2\big) {\rm d}t \wedge {\rm d} x _3 \wedge {\rm d} z \wedge {\rm d} \rho. \end{gather}
From \eqref{vgK}, we can infer that the volume of $(M, g _K)$ is finite, and the image of the moment map is a convex, pre-compact polytope in the Lie algebra $\mathfrak{t}$ of the torus $\mathbb{T} ^2$, cf.\ Figure \ref{fig:polytope}, which is the picture, taken from \cite[Section~8]{BG}, of the moment polytope of the Chen--Teo instanton. Notice that, apart from the dashed edge $E _{\infty}$, representing the boundary at infinity, each edge~$E _i$,~${i = 0, 1, \dots, r}$ is associated to the interval $(z _i, z _{i + 1})$ on the $z$-axis $\rho = 0$.

The moment with respect of $\omega _K$ of the Killing vector fields $\partial _t$ and $\partial _{x_3}$ -- which, in general, don't form a basis of $\Lambda$ -- are denoted by $x _1$ and $\mu$ respectively, with
\[
 x _1 = \frac{2}{H_z}, \qquad \mu = - \frac{1}{A} \frac{z H _z + \rho H_{\rho} - 2 H}{H_z}, \]
where, we recall, $H$ is defined by \eqref{H} \cite[Proposition 6.1]{BG}. Notice however that in general~$\partial _t$ and $\partial _{x_3}$, regarded as elements of the Lie algebra $\mathfrak{t}$ of the torus $\mathbb{T} ^2$ don't form a basis of the lattice~$\Lambda$ in $\mathfrak{t}$ induced by $\mathbb{T} ^2$. In restriction to the boundary $\rho = 0$, the moments are functions of $z$, with the following expressions on each interval $(z _i, z _{i + 1})$, where $f'_i \neq 0$,
\[
 x_1 = \frac{1}{f (z)}, \qquad \mu = - \frac{F_i}{f (z)} + \frac{1}{A} \left(\frac{f (z)}{f' (z)} - z\right). \]

The expression \eqref{g-harmark} of the metric $g$ holds on the open set, $M _0$, where $\rho \neq 0$, i.e., where the torus action is free. In the toric K\"ahler setting, the boundary $\rho = 0$ of this open set is formed of $(r - 1)$ compact invariant divisors, isomorphic to $2$-spheres, and of two divisors isomorphic to punctured spheres, corresponding to a point at infinity for each of them, encoded by the $r + 1$ edges of the moment polytope. To each edge of the moment polytope, itself encoded by an interval $(z_i, z_{i = 1})$, $i = 0, \dots, r$, is associated a Killing vector field, $v_i$, regarded as an element of~$\mathfrak{t}$, actually a primitive element of $\Lambda$: $v_i$ is then the generator of a $S^1$-action of period $2 \pi$, and vanishes on the corresponding invariant divisor. It follows from \eqref{V-0} that $v _i$ has the following form:
\begin{gather} v _i = f'_i (\partial _{x_3} + F _i \partial _t) \qquad \text{if}\ f'_i \neq 0, \qquad
 v _i = \frac{1}{A} f_i ^2 \partial _t \qquad \text{if}\ f'_i = 0. \label{v}
\end{gather}
 More generally, if the metric admits a conical singularity along the invariant divisor $E _i$, of angle~$2 \pi \alpha _i$, then
 \begin{gather}
 v _i = \alpha _i f'_i (\partial _{x_3} + F _i \partial _t) \qquad \text{if}\ f'_i \neq 0, \qquad
 v _i = \frac{1}{A} \alpha _i f_i ^2 \partial _t \qquad \text{if}\ f'_i = 0.\label{v-alpha}
\end{gather}

The conditions that $(M_0, g)$ will smoothly extend to the boundary, possibly with conical singularities of $g$ on the invariant divisors, is that each pair $v _i$, $v_{i + 1}$ be a basis of the lattice~$\Lambda$, i.e., that each pair be related to the next one by an element of the group ${\rm GL} (2, \mathbb{Z})$ of $2 \times 2$ matrices with integer coefficients and determinant equal to $ \pm 1$, i.e.,
\[ \begin{pmatrix} v_{i - 1} \\ v_i \end{pmatrix} = \begin{pmatrix} \ell _i & - \epsilon _i \\ 1 & 0 \end{pmatrix} \begin{pmatrix} \label{v-matrix} v _i \\ v _{i + 1} \end{pmatrix}, \]
hence
\[
\ell _i v _i = v_{i - 1} + \epsilon _i v _{i + 1}, \qquad i = 1, \dots, r - 1, \]
where the $\ell _i$ are integers and $\epsilon = \pm 1$, cf.\ \cite[Section~8]{BG}.

As already mentioned, these conditions turn out to be quite restrictive, in particular impose that $r$ cannot exceed $3$. For each value $1$, $2$ or $3$ of $r$, the only toric Hermitian ALF gravitational instantons are then as follows, cf.\ Theorems A and~8.2 in \cite{BG}:
\begin{itemize}\itemsep=0pt
\item $r = 1$: The {\it self-dual Taub-NUT instanton}, i.e., the Euclidean self-dual Taub-NUT on $\mathbb{R} ^4$, with the orientation opposite to the one induced by its hyperk\"ahler structure. Its piecewise affine function is $f (z) = 2 n + |z|$ and its generating function is $U (\rho, z) = 2 n \log{\rho ^2} + U _0 (\rho, z)$.
\item $r = 2$: \begin{itemize}\itemsep=0pt
\item[(i)] The {\it Taub-bolt instanton}, discovered by D. Page in 1978, whose piecewise affine function is $f (z) = 3 b + \frac{1}{2} |z + b| + \frac{1}{2} |z - b|$, $b = \frac{3}{4} |n|$.
 \item[(ii)] The {\it Euclidean Kerr metrics}, discovered by R. Kerr in 1963, with
 \[f (z) = m + \frac{1}{2} \left(1 - \frac{a}{b}\right) |z + b| + \frac{1}{2} \left(1 + \frac{a}{b}\right) |z - b|, \qquad 0 < |a| < b = \big(m ^2 + a^2\big) ^{\frac{1}{2}}.\]
 \item[(iii)] The {\it Euclidean Schwarzschild metric}, discovered by K. Schwarzschild in 1918, with $f (z) = m + \frac{1}{2} |z + m| + \frac{1}{2} |z - m|$, which can be viewed as a particular case of Euclidean Kerr metric, with $a = 0$ and $b = m$.
\end{itemize}
 \item $r = 3$: The 1-parameter family of {\it Chen--Teo instantons}, discovered by Yu Chen and Edward Teo in 2011, cf.\ \cite{CT2}, with
 \begin{gather} f (z) = \frac{1}{2} \big(1 - p ^{\frac{3}{2}} - q ^{\frac{3}{2}} + q |z + q ^{\frac{1}{2}} - q| + |z| + p |z - p ^{\frac{1}{2}} + p|\big), \label{f-CT}\\
 0 < p < 1, \qquad 0 < q < 1, \qquad p + q = 1, \label{pq-CT} \\
 f _1 = p q ^{\frac{1}{2}}, \qquad f_2 = p q, \qquad f_3 = p^{\frac{1}{2}} q.\label{fi-CT}
 \end{gather}
 In contrast with the previous cases, the Chen--Teo instantons are not the Euclidean form of Lorentzian spaces, and their anti-self-dual Weyl tensor $W ^-$ is {\it not} degenerate, as shown in \cite{AA}.
\end{itemize}

\begin{figure}[t]\centering
\begin{tikzpicture}
\draw[step = 1cm,blue,very thin] (-4,-1) grid (4,4);
\draw[very thick,-](-4,-1)--(4,-1);
\draw[fill] (-2,-1) circle[radius = 0.1];
\node[below] at (-2,-1) {$z_1 = q - q^{\frac{1}{2}}$};
\node[below] at (-2,0.4) {$f_1 = p q^{\frac{1}{2}}$};
\draw[fill] (1.5,-1) circle[radius = 0.1];
\node[below] at (1.7,-1) {$z_3 = p^{\frac{1}{2}} - p$};
\draw[fill] (1.5,0.7) circle[radius = 0.1];
\node[below] at (1.8,0.6) {$f _3 = p^{\frac{1}{2}} q$};
\draw[fill] (0,0) circle [radius = 0.1];
\node[below] at (-0,-0.1) {$f _2 = p q$};
\draw[thick,red, -] (1.5,0.7)--(4,3.2);
\draw[fill] (-2,0.5) circle[radius = 0.1];
\node[below] at (-2,0.5) {};
\draw[thick, red, -] (-2,0.5)--(0,0);
\draw[thick,red, -] (-4,2.5)--(-2,0.5);
\draw[thick,red, -] (0,0)--(1.5,0.7);
\node[above] at (0.5,0.4) {$f'_2 = q$};
\node[below] at (-2.3,1.9) {$f'_0 = -1$};
\node[below] at (-1,1) {$f'_1 = - p$};
\draw[fill] (0,-1) circle[radius = 0.1];
\node[below] at (0,-1.05) {$z_2 = 0$};
\node[above] at (2.1,1.6) {$f'_3 = 1$};
\end{tikzpicture}
\caption{The piecewise affine function of the Chen--Teo metric.}\label{fig:affine}
\end{figure}
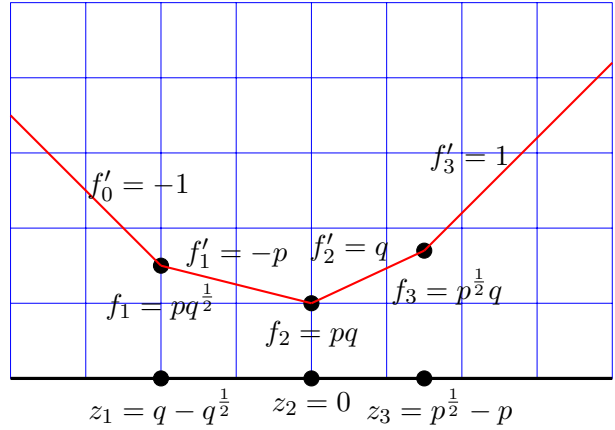

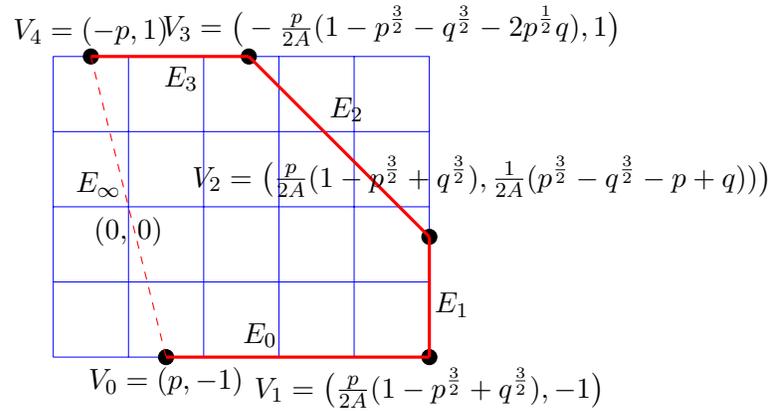
\begin{figure}[t]\centering
\begin{tikzpicture}
\draw[step=1cm, blue, very thin] (-2, -2) grid (3,2);
\draw[fill] (-1.5,2) circle [radius=0.1];
\draw[fill] (0.6,2) circle [radius=0.1];
\draw[fill] (-0.5,-2) circle [radius=0.1];
\draw[fill] (3,- 0.4) circle [radius=0.1];
\draw[fill] (3,-2) circle [radius=0.1];
\draw[dashed, red, -] (-1.5, 2) -- (-0.5, -2);
\draw[very thick, red, -] (-1.5, 2) -- (0.6, 2);
\draw[very thick, red, -] (3, -0.4) -- (3, -2);
\draw[very thick, red, -] (-0.5, -2) -- (3, -2);
\draw[very thick, red, -] (3,-0.4) -- (0.6,2);
\node [below] at (-1, 0) {(0, 0)};
\node [above] at (- 1.5, 2) {$V_4= (- p,1)$};
\node [above] at (2.5, 2) {$V_3 = \big(- \frac{p}{2 A} (1 - p ^{\frac{3}{2}} - q ^{\frac{3}{2}} - 2 p ^{\frac{1}{2}} q), 1\big)$};
\node [below] at (-0.5, -2) {$V_0 = (p,-1)$};
\node [below] at (3, -2) {$V_1 = \big(\frac{p}{2A} (1 - p^{\frac{3}{2}} + q ^{\frac{3}{2}}), - 1\big)$};
\node[above] at (3.7, - 0) {$V_2 = \big(\frac{p}{2 A} (1 - p^{\frac{3}{2}} + q ^{\frac{3}{2}}), \frac{1}{2 A} (p^{\frac{3}{2}} - q^{\frac{3}{2}} - p + q))\big)$};
\node[above] at (0.75,- 2) {$E_0$};
\node[above] at (3.3,- 1.6) {$E_1$};
\node[above] at (1.9,1) {$E_2$};
\node[below] at (- 0.3, 2) {$E_3$};
\node[above] at (- 1.4, 0) {$E_{\infty}$};
\end{tikzpicture}
\caption{The Chen--Teo moment polytope in the $x, y$-plane, with respect to the $\mathbb{Z}$-basis $v_1 = - p (\partial _{x_3} + F_1 \partial _t)$, with $F_0 = 0$, where $x = - p (y + F _1 x_1)$, $y = \mu + \frac{1}{2 A} \big(p^{\frac{3}{2}} - q ^{\frac{3}{2}} - p + q\big)$, and $2 A = 1 - p ^{\frac{3}{2}} - q ^{\frac{3}{2}}$. The slope of the edge between $V_2$ and $V_3$ is $- 1$ for any value of the parameter $p$.}
\label{fig:polytope}
\end{figure}

More generally, we shall consider smooth completions of the metric $g$ given by \eqref{g-harmark} admitting conical singularities along the invariant divisors, as described by Theorem 7.5 in \cite[Section~7]{BG}. As shown in \cite[Section~7]{BG}, this can be done for the whole family of {\it Kerr--Taub--NUT} family, introduced by G.W.~Gibbons and M.J.~Perry in 1980, which includes the instantons mentioned above when $r = 2$.

This also concerns the {\it Chen--Teo {\rm 4}-parameter family} introduced in \cite{CT3}, which we shall explore in the next section.

\subsection{The self-dual Eguchi--Hanson metric}\label{sec:self-dual-eguchi} The {\it Eguchi--Hanson metric} was first discovered by Tohru Eguchi and Andrew J. Hanson in \cite{EH}, and by Eugenio Calabi in \cite{calabi}; it is also a member of the Gibbons--Hawking family of hyperk\"ahler metrics \cite{GH}. Like the Taub-NUT metric quoted above, also a member of the Gibbons--Hawking family, the Eguchi--Hanson metric is of type $O^+ D^-$ with respect to the orientation determined by the hyperk\"ahler structure, meaning that $W^+ \equiv 0$, while $W^-$ is degenerate, but non-zero. With respect to the opposite
orientation it is then of type $D^+O^-$ and will then be called {\it the self-dual Eguchi--Hanson metric}. This can be written in Harmark form \eqref{g-harmark}, with \smash{$\rho = \big(r^2 - b^2\big)^{\frac{1}{2}} \sin{\theta}$}, $z = r \cos{\theta}$, $V = \frac{r}{r^2 - b^2}$, $F = - \cos{\theta}$, ${\rm e}^{2 \nu} = \frac{r}{r^2 - b^2 \cos ^2{\theta}}$. It can be shown that the simple eigenvalue of $W^+$ is $\lambda _+ = \frac{2 b ^2}{r ^3}$ and the conformal K\"ahler metric is then conveniently chosen to be $g _K = \frac{1}{r^2} g$, whose K\"ahler class is then $\omega _K = - \frac{{\rm d}r}{r^2} \wedge ({\rm d}t + \cos{\theta} {\rm d} x_3) - \frac{1}{r} \sin{\theta} {\rm d}\theta \wedge {\rm d} x_3$, so that the moment $x_1$ of $\partial _t$ be equal to $\frac{1}{r}$ and $k = 2 b^2$.
Unlike the self-dual Taub-NUT metric, the self-dual Eguchi--Hanson metric is ALE, not ALF, but its generating function, $U_{\rm EH}$, is nevertheless of the same type \eqref{U} as the generating functions of the gravitational instantons considered in this note, namely
\begin{align*} U_{\rm EH} (\rho, z) & = \frac{1}{2} U_0(\rho, z + b) + \frac{1}{2} U_0 (\rho, z - b) \\ & = d_1 - \frac{1}{2} (z + b) \log{\frac{d_1 + z + b}{d_1 - z - b}} + d_2 - \frac{1}{2} (z - b) \log{\frac{d_2 + z - b}{d_2 - z + b}}, \end{align*}
with $d _1 = \big(\rho ^2 + (z + b) ^2\big)^{\frac{1}{2}}$ and $d_2 = \big(\rho ^2 + (z - b) ^2\big)^{\frac{1}{2}}$, and the corresponding piecewise affine function is then
\[ f_{\rm EH} (z) = \frac{1}{2} |z + b| + \frac{1}{2} |z - b|. \]
It may be observed that the positive constant $A$ appearing in the general expression \eqref{f} is here equal to $0$ and that the identity $k = 2 A$ is here no longer valid, showing again that the self-dual Eguchi--Hanson metric does not belong to the family of gravitational instantons considered in this paper. It may however be viewed as a limit, as already observed by Don Page in \cite{P}, cf.~also \cite[Section~7.2]{BG}.
In the current setting, this can be viewed by considering the following one-parameter family of metrics encoded by their piecewise affine functions of the form
\[ f _A (z) = A + \frac{1}{2} |z + b| + \frac{1}{2} |z - b|, \]
 normalized by the condition $A + b = 1$, cf.\ Remark \ref{rem-change}, with $A, b \geq 0$; notice that $f_1 = f _2 = 1$, and that most metrics in this family have conical singularities along invariant divisors, of angles $2 \pi \alpha _i$, $i = 1, 2, 3$. The vector fields attached to the corresponding polytopes, cf.\ \eqref{v}--\eqref{v-alpha}, are then $v_0 = - \alpha _0 (\partial _{x_3} + F _0 \partial _t)$, $v _1 = \alpha _1 \frac{1}{A} \partial _t$, $v _2 = \alpha _2 (\partial _{x_3} + F _2 \partial _t)$, and the regularity condition is then:
 $\left(\begin{smallmatrix} v_0 \\ v_1\end{smallmatrix}\right) = \left(\begin{smallmatrix} \ell & - \epsilon \\ 1 & 0 \end{smallmatrix}\right) \left(\begin{smallmatrix} v_1 \\ v_2 \end{smallmatrix}\right)$, for some integer $\ell$ and $\epsilon = \pm 1$; we then have $\alpha _0 = \epsilon \alpha _2$, hence $\epsilon = 1$, $\alpha _0 = \alpha _2$, and we can actually assume $\alpha _0 = \alpha _2 = 1$, and $\ell \alpha = A (F _2 - F_0)$, by setting $\alpha _1 = \alpha$, hence, by \eqref{Fr-F0}, $\ell \alpha = 2 \big(A^2 - b^2\big) = 2 (A - b) = 2 (2 A - 1)$, where we can assume that $\ell$ is equal to~$1$, $0$ or $-1$. When $A$ runs in the open interval $(0, 1)$, the corresponding metric is smooth in the following three cases: $\big(A = \frac{3}{4}, b = \frac{1}{4}, \ell = 1\big)$, $\big(A = \frac{1}{2}, b = \frac{1}{2}, \ell = 0\big)$ and $\big(A = \frac{1}{4}, b = \frac{3}{4}, \ell = -1\big)$, corresponding to the ``positive'' Taub-bolt metric, the Schwarzschild metric and the ``negative'' Taub-bolt metric respectively.\footnote{The ``positive'' and the ``negative'' Taub-bolt metrics are actually the same metric on the same manifold, namely the complex projective plane $\mathbb{CP}^2$ with a deleted point, with however opposite orientations, hence two different conformal K\"ahler structures: the ``positive'' Taub-bolt has the natural orientation of the tautological line bundle $\mathcal{O}(-1)$ over $\mathbb{CP} ^1$, the ``negative'' one the natural orientation of the dual line bundle $\mathcal{O}(1)$. Similarly, the hyperk\"ahler Eguchi--Hanson metric lives on the oriented manifold $\mathcal{O} (- 2)$, while the self-dual Eguchi--Hanson lives on the dual line bundle $\mathcal{O}(2)$.}

 When $A\in\big(0,\frac12\big)$ we can take $\ell=-1$, the topology is that of the negative Taub-bolt metric (the total space of $\mathcal{O}(1)$), the angle $4\pi(1-2A)$ goes from $0$ when $A\rightarrow\frac12$ to $4\pi$ when $A\rightarrow0$, hence the metric tends to the pull-back, from $\mathcal{O}(2)$ to $\mathcal{O}(1)$, of the self-dual Eguchi--Hanson metric, with a conical singularity of angle $4 \pi$. When $A\in\big(\frac12,1\big)$, we have $\ell=1$, the topology is that of the positive Taub-bolt metric (the total space of $\mathcal{O}(-1)$), again the angle $4\pi(2A-1)$ goes from $0$ when $A\rightarrow\frac12$ to $4\pi$ when $A\rightarrow1$. The limit for $A=1$ is the Taub--Nut metric on~$\mathbb{R}^4$ which is generated by the function $f_1(z)=1+|z|$.

 There is a symmetry around $A=\frac12$: the metrics for $A=\frac12\pm a$ are the same with the orientation reversed, up to scale. So it may seem curious that the limits for $A=0$ and $A=1$ are the selfdual Eguchi--Hanson metric and the Taub-NUT metric. This contradiction is solved by understanding that these are limits at different scales: the Taub-NUT metric is obtained when $A\rightarrow1$ by shrinking the 2-sphere to a point, and by rescaling there is a bubble which is~$\mathcal{O}(-1)$ with the 2-cover of the Eguchi--Hanson metric. This is precisely what we see on the other side $A\rightarrow0$, with the opposite orientation.

Finally notice the change of topology and of orientation at $\big(A = \frac{1}{2}, \ell = 0, b = \frac{1}{2}\big)$, encoding the (Riemannian) Schwarzschild metric, which lives on the product $S^2 \times \mathbb{R} ^2$, with its natural two orientations.

\section{The Chen--Teo family}

The Chen--Teo 4-parameter family is actually relevant to the general treatment of the preceding section, i.e., is included and probably coincides with the family of toric Hermitian ALF gravitational instantons, with $r = 3$, when the $z$-axis admits 3 angular points, $z_1 < z _2 < z _3$.

The convex piecewise affine function $f$ as then the following general form:
\begin{gather} \label{f-3} f (z) = A + \frac{1}{2} (1 - p) |z - z _1| + \frac{1}{2} (p + q) |z - z_2| + \frac{1}{2} (1 - q) |z - z _3|, \end{gather}
where
\[ - 1 < - p < q < 1, \]
are the slopes of $f$, on the open intervals $(- \infty, z_1)$, $(z _1, z_2)$, $(z_2, z_3)$, $(z _3, \infty)$ respectively. The pair $(p, q)$ then belongs to the open domain of $\mathbb{R} ^2$ defined by
\[ 
 - 1 < p < 1, \qquad - 1 < q < 1, \qquad p + q > 0.\]
We denote $f _1 := f (z _1)$, $f_2 := f(z_2) $, $f_3 := f (z_3)$, and, in addition to $p$, $q$, we introduce two positive parameters $a$, $b$ by
\[
 a := f_1 ^2/f_2 ^2, \qquad b := f_3 ^2/f_2 ^2. \]
Alternatively,
\begin{gather} \label{ab-pq} \sqrt{a} - 1 = p \frac{(z_2 - z_1)}{f_2}, \qquad \sqrt{b} - 1 = q \frac{(z_3 - z_2)}{f_2}. \end{gather}
Then $a > 1$ if $p > 0$, $a <1$ if $p < 0$ and $a = 1$ if $p = 0$; similarly, $b > 0$ if $q > 0$, $b < 1$ if $q < 0$ and $b = 1$ if $q = 0$, and
\begin{gather} \label{lim} \lim _{p \to 0}{\frac{(a - 1)}{p}} = \frac{2 (z_2 - z_1)}{f_2}, \qquad \lim_{q \to 0}{\frac{(b - 1)}{q}} = \frac{2 (z_3 - z_2)}{f_2}. \end{gather}
Notice that the parameters $a$, $b$, as well as the parameters $p$, $q$, are insensitive to the transform described in Remark \ref{rem-change}.

From \eqref{f-3}, we get $f_2 = A + \frac{f_2}{2} \big(\frac{(\sqrt{a} - 1)}{p} (1 - p) + \frac{(\sqrt{b} - 1)}{q} (1 - q)\big)$, hence
\[
A = f_2 \frac{\big(p + q - \sqrt{a} q (1 - p) - \sqrt{b} p (1 - q)\big)}{2 p q} =
\frac{1}{2} \big(f _2 \big(\sqrt{a} + \sqrt{b}\big) - (z _3 - z_1)\big). \]

\begin{rem} \label{rem-normalised} For further use, it will be convenient to ``normalize'' the convex piecewise affine function $f (z)$, via the transform described in Remark \ref{rem-change}, in order that $z _2 = 0$ and $f _2 = 1$, hence $f_1 = \sqrt{a}$, $f_3 = \sqrt{b}$. The convex piecewise affine function $f (z)$ is then given by \eqref{f-3}, with
\begin{gather*}
 A = \frac{\sqrt{a} + \sqrt{b}}{2} - \frac{1}{2} \left(\frac{\sqrt{a} - 1}{p} + \frac{\sqrt{b} - 1}{q}\right), \qquad z _1 = \frac{1 - \sqrt{a}}{p}, \qquad z _2 = 0, \qquad z _3 = \frac{\sqrt{b} - 1}{q}. \end{gather*}
\end{rem}

\subsection{Regularity} As in the introduction, we call a metric \emph{regular} if on some smooth compactification it has at worst conical singularities. In order to test the regularity of these metrics, we introduce the angles $2 \pi \alpha _0$, $2 \pi \alpha _1$, $2 \pi \alpha _2$, $2 \pi \alpha _3$, attached to each divisor, where the $\alpha _i$ are all positive, and we consider the corresponding Killing vector fields, when $p \neq 0$, $q \neq 0$,
\begin{alignat}{3}
 & v _0 = - \alpha _0 \left(\frac{\partial}{\partial x _3} + F _0 \frac{\partial}{\partial t}\right), \qquad && v _1 = - p \alpha _1 \left(\frac{\partial}{\partial x _3} + F _1 \frac{\partial}{\partial t}\right),& \nonumber \\
& v _2 = q \alpha _2 \left(\frac{\partial}{\partial x _3} + F _2 \frac{\partial}{\partial t}\right), \qquad && v _3 = \alpha _3 \left(\frac{\partial}{\partial x _3} + F _3 \frac{\partial}{\partial t}\right),& \label{v-i}
 \end{alignat}
where, for $i = 0, 1, 2, 3$, $F _i$ denotes the (constant) value of $F$ in the interval $(z_i, z _{i + 1})$ on the axis~$\rho = 0$, cf.\ \cite[Lemma 7.1]{BG}. The regularity conditions are then
\[ \begin{pmatrix} v_0 \\ v_1 \end{pmatrix} = \begin{pmatrix} \ell _1 & - \epsilon _1 \\ 1 & 0 \end{pmatrix} \begin{pmatrix} v_1 \\ v_2 \end{pmatrix}, \qquad \begin{pmatrix} v_1 \\ v_2 \end{pmatrix} = \begin{pmatrix} \ell _2 & - \epsilon _2 \\ 1 & 0 \end{pmatrix} \begin{pmatrix} v_2 \\ v_3 \end{pmatrix},
\]
where $\ell _1$, $\ell _2$ are integers, and $\epsilon _1$, $\epsilon _2$ are equal to $\pm 1$, hence
\[ \ell _1 v_1 = v_0 + \epsilon _1 v_2, \qquad \ell _2 v_2 = v_1 + \epsilon _2 v_3, \]
or else, in view of \eqref{v-i},
\begin{gather} \ell _1 p \alpha _1 + \epsilon _1 q \alpha _2 = \alpha _0, \label{C3-1}\\
 p \alpha _1 + \ell _2 q \alpha _2 = \epsilon _2 \alpha _3,\label{C3-2} \\
 \ell _1 p \alpha _1 F_1 + \epsilon _1 q \alpha _2 F_2 = \alpha _0 F_0, \label{C3-3}\\
 p \alpha _1 F_1 + \ell _2 q \alpha _2 F_2 = \epsilon _2 \alpha _3 F_3. \label{C3-4}\end{gather}
In view of \eqref{C3-1}, in \eqref{C3-3} the $F_i$ may be replaced by $F_i + c$ for any constant $c$, and likewise in~\eqref{C3-4} in view of \eqref{C3-2}. Also recall, cf.\ \cite[Proposition 7.3]{BG}, that the $F_i$ are related by
\begin{align}
& F_1 - F_0 = - \frac{f_1 ^2}{A} \frac{(1 - p)}{p} = - \frac{f_2 ^2}{A} \frac{a (1 - p)}{p},\nonumber
\\& F_2 - F_1 = \frac{f_2^2}{A} \frac{(p + q)}{p q},\nonumber \\&
F_3 - F _2 = - \frac{f_3 ^2}{A} \frac{(1 - q)}{q} = - \frac{f_2^2}{A} \frac{b (1 - q)}{q}. \label{F-F} \end{align}
In particular,
\begin{gather} \label{F3-F0} A (F_3 - F _0) = \frac{f_2 ^2}{p q} \big(p + q - a q (1 - p) - b p (1 - q)\big). \end{gather}

As observed above, in view of \eqref{C3-1}--\eqref{C3-2}, \eqref{C3-3}--\eqref{C3-4} can be rewritten as
\begin{gather} \ell _1 p (F_1 - F_0) \alpha _1 + \epsilon _1 q (F_2 - F_0) \alpha _2= 0, \label{C3-3bis}\\
\epsilon _2 p (F_1 - F_3) \alpha _1 + \epsilon _2 \ell _2 q (F_2 - F_3) \alpha _2 = 0. \label{C3-4bis}
\end{gather}
Since $\alpha _1$ and $\alpha _2$ are both positive, it follows that
\[(F_2 - F_0) (F_1 - F_3) = \epsilon _1 \ell _1 \ell _2 (F_1 - F_0) (F_2 - F_3), \]
hence
\[ \frac{(F_2 - F_0) (F_1 - F_3)}{(F_1 - F _0) (F _2 - F_3)} = \epsilon _1 \ell _1 \ell _2, \]
or, equivalently,
\begin{gather} \label{pre-n} {\bf n} := \frac{(F_3 - F_0) (F_1 - F_2)}{(F_1 - F _0) (F _2 - F_3)} = \epsilon _1 \ell _1 \ell _2 - 1. \end{gather}
In view of \eqref{F-F}--\eqref{F3-F0}, ${\bf n}$, defined by \eqref{pre-n}, has the following expression:
\begin{gather} \label{bfn} {\bf n} = \frac{(p + q) (p + q - a q (1 - p) - b p (1 - q))}{a q (1 - p) b p (1 - q)}, \end{gather}
and will be called the {\it normalized total NUT-charge}, cf.\ \cite[Section~III.B]{CT3}. By \eqref{pre-n}, ${\bf n}$ is then an {\it integer}, whenever the metric is regular. Notice that ${\bf n} = 0$ if and only if $F _3 - F _0 = 0$, i.e., if and only if the metric is AF.
\begin{rem} \label{rem-ab} {\rm Notice that \eqref{bfn} can be rewritten as
\[{\bf n} = \frac{(p + q)}{a b (1 - p) (1 - q)} \left(a + b - \frac{(a - 1)}{p} - \frac{(b - 1)}{q}\right).\]
 It follows from \eqref{bfn} and \eqref{ab-pq} that ${\bf n}$ is well defined at $p = 0$ or $q = 0$ and that the quantity \smash{$a + b - \frac{(a - 1)}{p} - \frac{(b - 1)}{q}$} has the sign of ${\bf n}$. In particular, a regular metric is AF if and only if the parameters $p$, $q$, $a$, $b$ are related by \smash{$a + b - \frac{(a - 1)}{p} - \frac{(b - 1)}{q} = 0$}.
} \end{rem}

From \eqref{C3-1}--\eqref{C3-4}, we infer
\begin{gather} \alpha _1 = \frac{1}{\ell _1 p} \frac{(F_0 - F_2)}{(F_1 - F_2)} \alpha _0 = \frac{1}{\epsilon _2 p} \frac{(F_3 - F_2)}{(F_1 - F_2)} \alpha _3, \label{alpha1-F}\\
\alpha _2 = \frac{1}{\epsilon _1 q} \frac{(F_0 - F_1)}{(F_2 - F_1)} \alpha _0 = \frac{1}{\epsilon _2 \ell _2 q} \frac{(F_3 - F_1)}{(F_2 - F_1)} \alpha _3, \label{alpha2-F} \\
\alpha _3 = \frac{\epsilon _2}{\ell _1} \frac{(F_0- F_2)}{(F_3 - F_2)} \alpha _0 = \frac{\epsilon _2 \ell _2}{\epsilon _1} \frac{(F_0- F_1)}{(F_3 - F_1)} \alpha _0. \label{alpha3-F}\end{gather}
In view of \eqref{F-F}, we then get
\begin{gather} \label{alpha12} \alpha _1 = \epsilon _2 \frac{b (1 - q)}{(p + q)} \alpha _3, \qquad \alpha _2 =
\epsilon _1 \frac{a (1 - p)}{(p + q)} \alpha _0, \end{gather}
from which we infer
\begin{gather} \label{epsilon} \epsilon _1 = \epsilon _2 = 1. \end{gather}
It then follows that
\begin{gather} {\bf n} = \ell _1 \ell _2 - 1,\qquad
v _2 = \ell _1 v _1 - v _0, \qquad v_3 = \ell _2 v_2 - v_1 = {\bf n} v_1 - \ell _2 v _0, \label{v2v3}\end{gather}
so that $v _0 \wedge v _3 = {\bf n} v _0 \wedge v _1$, hence
\[ {\bf n} = \det{(v_0, v _3)}, \]
 since the pair $(v_0, v_1)$ is a basis of the lattice $\Lambda$.
From \eqref{alpha1-F}--\eqref{alpha3-F} and \eqref{F-F}, we easily infer that the integers $\ell _1$, $\ell _2$ can be rewritten as
\begin{gather} \ell _1 = \frac{(p + q - a q (1 - p))}{b p (1 - q)} \frac{\alpha _0}{\alpha _3} = \left(1 + \frac{a q (1 - p)}{p + q} {\bf n}\right) \frac{\alpha _0}{\alpha _3},\label{ell1-def} \\
\ell _2 = \frac{(p + q - b p (1 - q))}{a q (1 - p)} \frac{\alpha _3}{\alpha _0} = \left(1 + \frac{b p (1 - q)}{p + q} {\bf n}\right) \frac{\alpha _3}{\alpha _0}.\label{ell2-def} \end{gather}
From \eqref{alpha12} and \eqref{epsilon}, the conical parameters $\alpha _1$, $\alpha _2$ are given by
\begin{gather} \label{alpha12-def} \alpha _1 = \frac{b (1 - q)}{p + q} \alpha _3, \qquad
 \alpha _2 = \frac{a (1 - p)}{p + q} \alpha _0, \end{gather}
while the relations \eqref{C3-1}, \eqref{C3-2}, \eqref{C3-3bis} and \eqref{C3-4bis} are expressed by
\begin{gather} \ell _1 p \alpha _1 + q \alpha _2 = \alpha _0, \label{R1}\\
 p \alpha _1 + \ell _2 q \alpha _2 = \alpha _3, \label{R2}\\
 - \ell _1 a p (1 - p) \alpha _1 + (p + q - a q (1 - p)) \alpha _2 = 0, \label{R3}\\
 (p + q - b p (1 - q)) \alpha _1 - \ell _2 b q (1 - q) \alpha _2 = 0.\label{R4} \end{gather}
By using the expressions of $\alpha _1$, $\alpha _2$ given by \eqref{alpha12-def}, it is easily checked that these relations are all satisfied.

 So far, we assumed that $p q \neq 0$. In view of \eqref{lim}, the cases when $p = 0$, $q > 0$ or $q = 0$, $p > 0$ are then obtained by continuity. When $p$ tends to $0$, then $q > 0$, since $p + q > 0$, and
\begin{gather*} \ell _1 = (1 + {\bf n}) \frac{\alpha _0}{\alpha _3}, \qquad \ell _2 = \frac{\alpha _3}{\alpha _0}, \qquad
 \alpha _1 = \frac{b (1 - q)}{q} \alpha _3, \qquad \alpha _2 = \frac{1}{q} \alpha _0, \end{gather*}
and
\begin{gather*} {\bf n} = \frac{(1 + q)}{b (1 - q)} - \frac{q}{b (1 - q)} \lim _{p \to 0}{\frac{(a - 1)}{p}} - 1. \end{gather*}

Similarly, when $q$ tends to $0$, then $p > 0$ and
\begin{gather*} \ell _1 = \frac{\alpha _0}{\alpha _3}, \qquad \ell _2 = (1 + {\bf n}) \frac{\alpha _3}{\alpha _0}, \qquad
 \alpha _1 = \frac{1}{p} \alpha _3, \qquad \alpha _2 = \frac{a (1 - p)}{p} \alpha _0 \end{gather*}
and
\begin{gather*} {\bf n} = \frac{(1 + p)}{a (1 - p)} - \frac{p}{a (1 - p)} \lim _{q \to 0}{\frac{(b - 1)}{q}} - 1. \end{gather*}

Recall that a metric of the Chen--Teo 4-parameter family is said to be {\it regular} if it can~be smoothly compactified along the invariant divisors, $D_i$, encoded by the edges $E_i$ of the momentum polytope, $i = 0, \dots, r$, with a suitable choice of conical singularities of angles $2 \pi \alpha _i$ along each~$D_i$. In view of the above, this happens if and only if we can choose $\alpha _0$, $\alpha _1$, $\alpha _2$,~$\alpha _3$, all positive, satisfying the conditions~\eqref{R1}--\eqref{R4}, hence, equivalently, the conditions~\eqref{ell1-def}--\eqref{alpha12-def}, in fact \eqref{ell1-def}--\eqref{ell2-def} only, since $\alpha _1$ and $\alpha _2$ are then be defined by~\eqref{alpha12-def}. We first observe that, in this case, it follows from \eqref{alpha12-def} that $\alpha _1$ and $\alpha _2$ are completely determined by $\alpha _0$ and $\alpha _3$, as $p + q > 0$; moreover, only the quotient $\alpha _0/\alpha _3$ is relevant, so that we can arrange that, say, $\alpha _0 = 1$.
 This being understood, we can formulate the following statement:
\begin{thm} \label{thm-regular}
Let $(M, g)$ be an element of the $4$-parameter Chen--Teo family, of parameter $p$, $q$, $a$, $b$; let ${\bf n}$ be the total NUT-charge of $g$:
\begin{itemize}\itemsep=0pt
 \item[$(1)$] $(M, g)$ is regular $($that is has a smooth compactification on which it has at worst conical singularities$)$ if and only if ${\bf n}$ is an integer.
\item[$(2)$] If $(M, g)$ is regular, then the boundary at infinity is diffeomorphic to $L$, where $L$ is
 \begin{itemize}\itemsep=0pt
 \item[$(i)$] a lens space of type $\ell/{\bf n}$, where $\ell$ is a factor of ${\bf n} + 1$, if ${\bf n} \neq - 1$;
 \item[$(ii)$] the sphere $S ^3$, if ${\bf n} = - 1$; \item[$(iii)$] $S^1 \times S^2$, if ${\bf n} = 0$.
\end{itemize}
\end{itemize}
\end{thm}

\begin{proof} (i) We already know that ${\bf n}$ is an integer whenever $g$ is regular. For the converse, in view of the above, we simply have to show that if ${\bf n}$ is an integer there always exist $\alpha _0$, $\alpha _3$ positive, in fact only $\alpha _3 > 0$ if we choose $\alpha _0 = 1$, satisfying the conditions \eqref{ell1-def}--\eqref{ell2-def}, where $\ell _1$, $\ell _2$ is some pair of integers such that ${\bf n} = \ell _1 \ell _2 - 1$. From \eqref{bfn}, we infer that
\begin{gather} \left(1 + \frac{a q (1 - p)}{p + q} {\bf n}\right) \left(1 + \frac{b p (1 - q)}{p + q} {\bf n}\right) = \ell _1 \ell _2 = {\bf n} + 1, \label{prod}\\
 \left(1 + \frac{a q (1 - p)}{p + q} {\bf n}\right) = \frac{p + q - a q (1 - p)}{b p (1 - q)}, \label{eq-a} \end{gather}
and
\begin{gather} \label{eq-b} \left(1 + \frac{b p (1 - q)}{p + q} {\bf n}\right) = \frac{p + q - b p (1 - q)}{a q (1 - p)}. \end{gather}

If ${\bf n} \geq 0$, hence $\ell _1 \ell _2 > 0$, it follows from \eqref{prod} that either $\big(1 + \frac{a q (1 - p)}{p + q} {\bf n}\big)$ and $\big(1 + \frac{b p (1 - q)}{p + q} {\bf n}\big)$ are both positive or both negative; the second case is in fact excluded, due to the constraints on the parameters $p$, $q$, $a$, $b$: indeed, since ${\bf n} \geq 0$, if $p >0$, $q >0$, then $\big(1 + \frac{a q (1 - p)}{p + q} {\bf n}\big)$ and~$\big(1 + \frac{b p (1 - q)}{p + q} {\bf n}\big)$ are clearly both positive, and this is still the case if $p \geq 0$, $q < 0$, because of~\eqref{eq-a}, or if $p < 0$, $q \geq 0$, because of \eqref{eq-b}. Thus, $\ell _1$, $\ell _2$ are both positive, and $\alpha _3$ is then defined by \eqref{ell1-def}--\eqref{ell2-def}.

If ${\bf n} < - 1$, hence $\ell _1 \ell _2 < 0$, then either $\big(1 + \frac{a q (1 - p)}{p + q} {\bf n}\big) > 0$ and $\big(1 + \frac{b p (1 - q)}{p + q} {\bf n}\big) < 0$ or vice versa. In the former case, we can choose $\ell _1 > 0$, $\ell _2 < 0$, in the latter case, chose instead $\ell _1 < 0$, $\ell _2 > 0$, and, in both cases, define $\alpha _3$ by $\alpha _3 = \frac{p + q - a q (1 - p)}{\ell _1 b p (1 - q)} = \frac{\ell _2 a q (1 - p)}{p + q - b p (1 - q)}$.

If ${\bf n} = - 1$, then $\ell _1 \ell _2 = 0$, so that either $\ell _1 = 0$ or $\ell _2 = 0$ or both. In the former case, we can define $\alpha _3 = \frac{\ell _2 (p + q)}{p + q - b p (1 - q)}$, where $\ell _2$ may be any integer of the sign of $p + q - b p (1 - q)$, and likewise if $\ell _2 = 0$. The most interesting case is when $\ell _1 = \ell _2 = 0$, i.e., when $p + q = a q (1 - p) = b p (1 - q)$; then, we can choose $\alpha _3 = \alpha _0 = 1$, $a = \frac{p + q}{q (1 - p)}$, $b = \frac{p + q}{p (1 - q)}$, $\alpha _1 = \frac{1}{p}$, $\alpha _2 = \frac{1}{q}$.

(ii) If $(M, g)$ is regular, we know by \eqref{v2v3} that $v _3 = {\bf n} v _1 - \ell _2 v _0$. According to Proposition 4.1 in \cite{BG}, the function $z + i \rho$ identifies the interior of the moment polytope $P$, equipped with the complex structure induced by $g$, with the Poincar\'e upper half-plane. At infinity, the topology of~$(M, g)$ is then $\mathbb{R} \times L$, where $L$ is obtained, from the product $[0, 1] \times \mathbb{T} ^2$, by identifying the circle $\{0\} \times S^1$ with the circle $\{1\} \times S^1$ via the rotation $2 i \pi \frac{\ell _2}{{\bf n}}$, where $\{0\} \times S^1$ encodes the orbit of $v _0$, around $E _0$, and $\{1\} \times S^1$ the orbit of $v _3 = {\bf n} v _1 - \ell _2 v _0$, around $E _3$, cf.\ \cite{hatcher,saveliev} and Figure~\ref{fig:lens}.
\end{proof}

 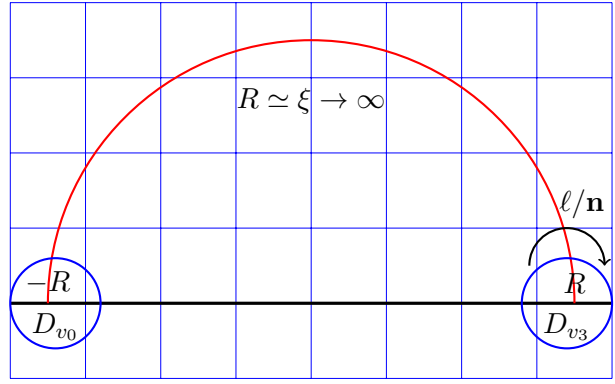
\begin{figure}[t]\centering
\begin{tikzpicture}
\draw[step = 1cm,blue,very thin] (-4,-1) grid (4,4);
\draw[very thick,-](-4,0)--(4,0);
\draw[thick, red] (3.5,0) arc (0:180:3.5);
\draw[thick,blue] (-3.4,0) ellipse (0.6 and 0.6);
\draw[thick,blue] (3.4,0) ellipse (0.6 and 0.6);
\node[below] at (0,3) {$R \simeq \xi \to \infty$};
\node[below] at (-3.4,0) {$D_{v _0}$};
\node[below] at (3.4,0) {$D_{v _3}$};
\draw[thick,black,<-] (3.9,0.5) arc (0:180:0.5);
\node[above] at (3.6,1) {$\ell/{\bf n}$};
\node[above] at (-3.5,0) {$- R$};
\node[above] at (3.5,0) {$R $};
\end{tikzpicture}
\caption{Lens space at infinity, obtained ``by attaching two solid tori $S^1 \times D^2$ together by a diffeomorphism $S^1 \times \partial D^2 \to S^1 \times \partial D^2$ sending a meridian $\{x\} \times \partial D^2$ to a circle of slope $\ell/{\bf n}$'', cf.\ \cite{hatcher}. The disk~$D _{v _0}$, resp.\ $D _{v_3}$, is formed by the orbits of the Killing vector field~$v_0$, resp.~$v _3$. The red half-circle is the hyperbolic geodesic in the Poincar\'e upper half-plane relating $- R$ to $R$ on the real axis, and $R$ tends to infinity.}
\label{fig:lens}
\end{figure}

\subsection{The case when the metric is smooth} \label{ss-smooth} If $\alpha _0 = \alpha _1 = \alpha _2 = \alpha _3 = 1$, i.e., if $(M, g)$ is a gravitational instanton, the system \eqref{R1}--\eqref{R4} becomes
\begin{gather} \ell _1 p + q = 1, \qquad p + \ell _2 q = 1, \label{smooth-1}\\
 p + q - a (1 - p) = 0, \qquad p + q - b (1 - q) = 0. \label{smooth-2}\end{gather}
By \eqref{smooth-1}, $p + q = 1 - (\ell _1 - 1) p = 1 - (\ell _2 - 1) q$. From \eqref{smooth-2}, we then infer
\[\frac{a - 1}{p} = \frac{2 - \ell _1}{1 - p}, \qquad \frac{b - 1}{q} = \frac{2 - \ell _2}{1 - q}. \]
It follows that $\ell _1 < 2$ and $\ell _2 < 2$. From \eqref{smooth-1}, we infer that $\ell _1 p = 1 - q > 0$ and $\ell _2 q = 1 - p > 0$. If $p > 0$ and $q > 0$, hence $\ell _1 = \ell _2 = 1$, we thus get ${\bf n} = 0$, $p + q = 1$, $a = \frac{1}{q}$, $b = \frac{1}{p}$, which characterizes the Chen--Teo instanton, cf.\ \eqref{f-CT}--\eqref{fi-CT}.
If $p > 0$ and $q < 0$, then $\ell _1 > 0$, hence $\ell _1 = 1$, so that $p + q = 1$ and $q = \ell _2 q$, which is impossible, since $\ell _2 q > 0$. Similarly, we cannot have $p < 0$ and $q > 0$. We thus recover the fact, already established in \cite[Section~7]{BG}, that the only toric Hermitian ALF gravitational instantons with 3 angular points are the Chen--Teo gravitational instantons.

\subsection{Some particular cases in the general ALF case} \label{sec:some-part-cases}

 If $\alpha _0 = \alpha _3 = 1$, then, by \eqref{R1}--\eqref{R2}, we get
\[ \ell _1 p \alpha _1 + q \alpha _2 = 1, \qquad p \alpha _1 + \ell _2 q \alpha _2 = 1, \]
while, by \eqref{alpha12-def}, we have
\[ a = \frac{(p + q) \alpha _2}{1 - p}, \qquad b = \frac{(p + q) \alpha _1}{1 - q}. \]
From \eqref{ell1-def}--\eqref{ell2-def}, we also infer
\begin{gather*} \ell _1 = \frac{1 - q \alpha _2}{p \alpha _1} = 1 + q \alpha _2 {\bf n}, \qquad \ell _2 = \frac{1 - p \alpha _1}{q \alpha _2} = 1 + p \alpha _1 {\bf n}, \\
 {\bf n} = \frac{\ell _1 - 1}{q \alpha _2} = \frac{\ell _2 - 1}{p \alpha _1} = \frac{1 - p \alpha _1 - q \alpha _2}{p q \alpha _1 \alpha _2}, \end{gather*}
and\
\[ \frac{a - 1}{p} = \frac{1 - \ell _1 \alpha _1 - \alpha _2}{1 - p}, \qquad \frac{b - 1}{q} = \frac{1 + \alpha _1 - \ell _2 \alpha _2}{1 - q}. \]

\begin{pc} \label{pc-alpha-032} We first consider the case when, in addition to $\alpha _0 = \alpha _3 = 1$, we suppose that $\alpha _2 = 1$, and we then put: $\alpha _1 = \alpha$ (similar developments can be done, by simply swapping~$p$ and $q$ if we suppose instead that $\alpha _1 = 1$ and $\alpha _2 = \alpha$). We then have
\begin{gather} \label{1-alpha} \ell _1 p \alpha + q = 1, \qquad p \alpha + \ell _2 q = 1, \end{gather}
and
\begin{gather*} a = \frac{p + q}{1 - p}, \qquad b = \frac{(p + q) \alpha}{1 - q},\\
\ell _1 = 1 + q {\bf n} = \frac{1 - q}{p \alpha}, \qquad \ell _2 = 1 + p \alpha {\bf n} = \frac{1 - p \alpha}{q},\\
 {\bf n} = \frac{\ell _1 - 1}{q} = \frac{\ell _2 - 1}{p \alpha} = \frac{1 - p \alpha - q}{p q \alpha}, \end{gather*}
and
\[\frac{a - 1}{p} = \frac{2 - \ell _1 \alpha}{1 - p}, \qquad \frac{b - 1}{q} = \frac{1 + \alpha - \ell _2}{1 - q}. \]
 \end{pc}

Interesting 1-parameter families are obtained by taking $q = 0$, from which we infer: $p > 0$, hence $\frac{1}{2} < p < 1$ -- since we have then $a > 1$; from \eqref{1-alpha} we also get $p \alpha = 1$, hence $1 < \alpha < 2$, and $\ell _1 = 1$, hence ${\bf n} = \ell _2 - 1$; we also infer: $a = \frac{p}{1 - p}$, $b = 1$, $\frac{a - 1}{p} = \frac{2 - \alpha}{1 - p}$ and $\frac{b - 1}{q} = 1 + \alpha - \ell_2$. For any ${\bf n} = \ell _2 - 1$, we thus get a 1-parameter family of regular metrics parametrised either by~$p \in (\frac{1}{2}, 1)$ or, equivalently, by the angle $2\pi\alpha \in (2\pi, 4\pi)$.

When $p$ tends to $1$, i.e., when $\alpha$ tends to 1, for any $\ell _2$, $a$ tends to $+ \infty$, $b = 1$ and \smash{$\frac{\sqrt{b} - 1}{q}$} tends to \smash{$\frac{2 - \ell _2}{2}$}; in view of Remark \ref{rem-normalised}, the metric then tends to the metric encoded by the affine piecewise function
\[
f ^{\alpha = 1} (z) = 1 - \frac{(2 - \ell _2)}{4} + \frac{1}{2} \Big|z + \frac{(2 - \ell _2)}{4}\Big| + \frac{1}{2} \Big|z - \frac{(2 - \ell _2)}{4}\Big|. \]

When $p$ tends to $\frac{1}{2}$, i.e., when $\alpha$ tends to $2$, for any $\ell _2$, $a = 1$, implying that $z _1 = z _2$ and~$\frac{\sqrt{a} - 1}{p} = 0$, and $\frac{\sqrt{b} - 1}{q}$ tends to $\frac{3 - \ell _2}{2}$. In view of Remark \ref{rem-normalised} again, the metric then tends to the metric encoded by the piecewise affine function
\[
 f ^{\alpha = 2} (z) = 1 - \frac{(3 - \ell _2)}{4} + \frac{1}{2} \Big|z + \frac{(3 - \ell _2)}{4}\Big| + \frac{1}{2} \Big|z - \frac{(3 - \ell _2)}{4}\Big|. \]
This limit when the angle goes to $4\pi$ corresponds to the process described in \cite[Section~9]{BG} where the $S^2$ with the conical singularity disappears at the limit $4\pi$ with a bubble which should be the 2-cover of the self-dual Eguchi--Hanson metric (see the family of Section~\ref{sec:self-dual-eguchi} for an example of this phenomenon).

By successively considering the particular cases when $\ell _2 = 2$, ${\bf n} = 1$, $\ell _2 = 0$, ${\bf n} = - 1$, $\ell _2 = - 1$, ${\bf n} = - 2$ and the AF case $\ell _2 = 1$, ${\bf n} = 0$, we thus get the following 1-parameter families.
\begin{itemize}\itemsep=0pt
 \item[(i)] $\ell _2 = 2$, ${\bf n} = 1$: $f^{\alpha = 1} (z) = 1 + |z|$, which encodes the self-dual Taub-NUT gravitational instanton, and $f ^{\alpha = 2} (z) = \frac{3}{4} + \frac{1}{2} |z + \frac{1}{4}| + \frac{1}{2} |z - \frac{1}{4}|$, which encodes the positive Taub-bolt metric.
\item[(ii)] $\ell _2 = 0$, ${\bf n} = - 1$: $f^{\alpha = 1} (z) = \frac{1}{2} + \frac{1}{2} |z + \frac{1}{2}| + \frac{1}{2} |z - \frac{1}{2}|$, which encodes the Schwarzschild gravitational instanton, and $f^{\alpha = 2} (z) = \frac{1}{4} + \frac{1}{2} |z + \frac{3}{4}| + \frac{1}{2} |z - \frac{3}{4}|$, which encodes the negative Taub-bolt metric.
\item[(iii)] $\ell _2 = - 1$, ${\bf n} = - 2$: $f^{\alpha = 1} (z) = \frac{1}{4} + \frac{1}{2} |z + \frac{3}{4}| + \frac{1}{2} |z - \frac{3}{4}|$, which encodes the negative Taub-bolt gravitational instanton, and $f^{\alpha = 2} (z) = \frac{1}{2} |z + 1| + \frac{1}{2} |z - 1|$, which encodes the self-dual Eguchi--Hanson metric.
\item[(iv)] $\ell _2 = 1$, ${\bf n} = 0$ (this is an AF case): $f^{\alpha = 1} (z) = \frac{3}{4} + \frac{1}{2} |z + \frac{1}{4}| + \frac{1}{2} |z - \frac{1}{4}|$, which encodes the positive Taub-bolt gravitational instanton, and $f^{\alpha = 2} (z) = \frac{1}{2} + \frac{1}{2} |z + \frac{1}{2}| + \frac{1}{2} |z - \frac{1}{2}|$, which encodes the Schwarzschild metric.
\end{itemize}

\subsection{The AF case} \label{ss-AF}
As mentioned above, the normalized total NUT-charge ${\bf n}$ is equal to zero if and only if the metric is AF. We then have $\ell _1 \ell_2 = 1$, while it follows from \eqref{ell1-def}--\eqref{ell2-def} that $\ell _1 = \ell _2 ^{-1} = \frac{\alpha _0}{\alpha _3}$. Since $\ell _1$ and $\ell _2$ are both positive integers, we eventually infer that
\[ \ell _1 = \ell _2 = 1, \]
so that
\[ \alpha _0 = \alpha _3. \]
The conditions \eqref{R1}--\eqref{R4} then become
\begin{gather*}
p \alpha _1 + q \alpha _2 = \alpha _0 = \alpha _3, \\
 - a p (1 - p) \alpha _1 + (p + q - a q (1 - p)) \alpha _2 = 0, \\
 (p + q - b p (1 - q)) \alpha _1 - b q (1 - q) \alpha _2 = 0. \end{gather*}

Without loss of generality, we can suppose that $\alpha _0 = \alpha _3 = 1$. We thus get
\[ p \alpha _1 + q \alpha _2 = 1, \]
as well as
\[
 a = \frac{(p + q)}{1 - p} \alpha _2, \qquad b = \frac{(p + q)}{1 - q} \alpha _1. \]

 \begin{pc} \label{pc-AF1} An interesting case is when $\alpha _1 = \alpha _2 =: \alpha > 0$. This happens if and only~if
\[ a = \frac{1}{1 - p}, \qquad b = \frac{1}{1 - q}, \]
 and then
 \[ \alpha = \frac{1}{p + q}. \]
 If $p + q = 1$, i.e., if $\alpha$ tends to $1$, we thus recover the Chen--Teo instanton. If, however, $p + q$ tends to $2$, i.e., if both $p$ and $q$ tend to $1$, then $\alpha$ tends to $\frac{1}{2}$, while the normalised piecewise affine function, cf.\ Remark \ref{rem-normalised}, tends to $f (z) = 1 + |z|$; we thus obtain a quotient by $\mathbb{Z}/2 \mathbb{Z}$ of the self-dual Taub-NUT space.
Finally, if $p + q$ tends to $0$, i.e., $p$ and $q$ both tend to $0$ and $\alpha$ then tends to $+ \infty$, then $a$ and $b$ both tend to $1$, $\frac{\sqrt{a} - 1}{p} = \frac{\sqrt{b} - 1}{q} = \frac{1}{2}$, and the piecewise affine function tends to $f_{\rm EH} (z) = \frac{1}{2} |z + \frac{1}{2}| + \frac{1}{2} |z - \frac{1}{2}|$, which encodes the self-dual Eguchi--Hanson metric, cf.~Section~\ref{sec:self-dual-eguchi}.
\end{pc}

 \begin{pc} \label{pc-AF2} An interesting case is with $\alpha _0 = \alpha_2 = \alpha _3 = 1$, and we then put: $\alpha _1 =: \alpha$. We thus get
 \[ a = \frac{p + q}{1 - p}, \qquad b = \frac{p + q}{p} \qquad \text{and}\qquad \alpha = \frac{1 - q}{p}. \]

 Interesting 1-parameter families are obtained by fixing the parameter $q$ (this actually amounts to fixing the asymptotic behavior of the metric). Then $\frac{1-q}2<p<1$ (the first inequality coming from $a>1$). We get a family of AF examples which are smooth except for one conical singularity along a $S^2$ with angle $2\pi\alpha\in(2\pi(1-q),4\pi)$; $q$ being fixed, this family is parametrised either by~$p$, or by $\alpha$, or, better by $\tau := \alpha - 1$, so that $- q < \tau < 1$. In view of Remark \ref{rem-change}, for each value of $\tau$ it is easy to check that the corresponding metric is encoded by the following piecewise affine function
 \begin{align*} f ^{q, \tau} (z) ={} & \frac{(1 + q \tau)^{\frac{1}{2}}}{2 q (1 - q)} \big((1 + q \tau)^{\frac{1}{2}} - q (q + \tau)^{\frac{1}{2}} - (1 - q) ^{\frac{3}{2}}\big) \\ & + \frac{(q + \tau)}{2 (1 + \tau)} \bigg|z + \frac{(1 - \tau ^2)}{(q + \tau) ^{\frac{1}{2}} \big((1 + q \tau) ^{\frac{1}{2}} + (q + \tau)^{\frac{1}{2}}\big)}\bigg| + \frac{(1 + q \tau)}{2 (1 + \tau)} |z| \\ & + \frac{(1 - q)}{2} \bigg\vert z - \frac{(1 + \tau)}{(1 - q) ^{\frac{1}{2}} \big((1 + q \tau)^{\frac{1}{2}} + (1 - q)^{\frac{1}{2}}\big)}\bigg\vert.
\end{align*}
 When $\tau = 0$, i.e., $\alpha = 1$ and $p = 1 - q$, the corresponding metric is smooth: it is the Chen--Teo gravitational instanton of parameter $q, p = 1 - q$, whose piecewise affine function is
 \begin{gather*} f ^{\tau = 0} (z) = \frac{1}{2 p q} \big(1 - p ^{\frac{3}{2}} - q ^{\frac{3}{2}}\big) + \frac{q}{2} \bigg|z + \frac{1}{q ^{\frac{1}{2}} (1 + q ^{\frac{1}{2}})}\bigg| + \frac{1}{2} |z| + \frac{p}{2} \bigg|z - \frac{1}{p ^{\frac{1}{2}} (1 + p ^{\frac{1}{2}})}\bigg|. \end{gather*}
 If $\tau = -q$, i.e., $\alpha = 1 - q$, we get
 \[ f ^{\tau = - q} (z) = \frac{(1 + q)^{\frac{1}{2}}}{2 q} \big((1 + q) ^{\frac{1}{2}} - (1 - q)\big) + \frac{(1 - q)}{2} |z|
 + \frac{(1 - q)}{2} \bigg|z - \frac{1}{1 + (1 + q) ^{\frac{1}{2}}}\bigg|. \]
 If $\tau = 1$, i.e., $\alpha = 2$, we get
 \begin{align*} f^{\tau = 1} (z) ={} & \frac{(1 + q) ^{\frac{1}{2}}}{2 q} \big((1 + q) ^{\frac{1}{2}} - (1 - q)^{\frac{1}{2}}\big) + \frac{(1 + q)}{2} \vert z\vert \\ & + \frac{(1 - q)}{2} \bigg\vert z - \frac{2}{(1 - q) ^{\frac{1}{2}} \big((1 + q) ^{\frac{1}{2}} + (1 - q)^{\frac{1}{2}}\big)}\bigg\vert.
 \end{align*}
 The limit for the angle $4\pi$ is again obtained by blowing down the $S^2$ to a point and is a Kerr metric. The limit for the angle $2\pi(1-q)$ is a Kerr--Taub-bolt metric with a conical singularity: it changes the topology at infinity because there is a bubble at infinity. The special case $q=0$ was already studied in Section~\ref{sec:some-part-cases}, case~(iv).
\end{pc}

\subsection*{Acknowledgements}
We thank the referees for their careful reading of the article.

\pdfbookmark[1]{References}{ref}
\LastPageEnding


\begin{thebibliography}{99}
\footnotesize\itemsep=0pt

\bibitem{AA}
Aksteiner S., Andersson L., Gravitational instantons and special geometry,
 \href{https://arxiv.org/abs/2112.11863}{arXiv:2112.11863}.

\bibitem{ACG}
Apostolov V., Calderbank D.M.J., Gauduchon P., Ambitoric geometry {I}:
 {E}instein metrics and extremal ambik\"ahler structures, \href{https://doi.org/10.1515/crelle-2014-0060}{\textit{J.~Reine
 Angew. Math.}} \textbf{721} (2016), 109--147, \href{https://arxiv.org/abs/1302.6975}{arXiv:1302.6975}.

\bibitem{BG}
Biquard O., Gauduchon P., On toric {H}ermitian {ALF} gravitational instantons,
 \href{https://doi.org/10.1007/s00220-022-04562-z}{\textit{Comm. Math. Phys.}} \textbf{399} (2023), 389--422,
 \href{https://arxiv.org/abs/2112.12711}{arXiv:2112.12711}.

\bibitem{calabi}
Calabi E., M\'etriques k\"ahl\'eriennes et fibr\'es holomorphes, \href{https://doi.org/10.24033/asens.1367}{\textit{Ann.
 Sci. \'Ecole Norm. Sup.}} \textbf{12} (1979), 269--294.

\bibitem{CT1}
Chen Y., Teo E., Rod-structure classification of gravitational instantons with
 {$U(1)\times U(1)$} isometry, \href{https://doi.org/10.1016/j.nuclphysb.2010.05.017}{\textit{Nuclear Phys.~B}} \textbf{838} (2010),
 207--237, \href{https://arxiv.org/abs/1004.2750}{arXiv:1004.2750}.

\bibitem{CT2}
Chen Y., Teo E., A new {AF} gravitational instanton, \href{https://doi.org/10.1016/j.physletb.2011.07.076}{\textit{Phys. Lett.~B}}
 \textbf{703} (2011), 359--362, \href{https://arxiv.org/abs/1107.0763}{arXiv:1107.0763}.

\bibitem{CT3}
Chen Y., Teo E., Five-parameter class of solutions to the vacuum {E}instein
 equations, \href{https://doi.org/10.1103/PhysRevD.91.124005}{\textit{Phys. Rev.~D}} \textbf{91} (2015), 124005, 17~pages,
 \href{https://arxiv.org/abs/1504.01235}{arXiv:1504.01235}.

\bibitem{derdzinski}
Derdzi\'nski A., Self-dual {K}\"ahler manifolds and {E}instein manifolds of
 dimension four, \textit{Compos. Math.} \textbf{49} (1983), 405--433.

\bibitem{dixon}
Dixon K., Regular ambitoric 4-manifolds: from {R}iemannian {K}err to a complete
 classification, \href{https://doi.org/10.4310/CAG.2021.v29.n3.a3}{\textit{Comm. Anal. Geom.}} \textbf{29} (2021), 629--679,
 \href{https://arxiv.org/abs/1604.03156}{arXiv:1604.03156}.

\bibitem{EH}
Eguchi T., Hanson A.J., Asymptotically flat self-dual solutions to {E}uclidean
 gravity, \href{https://doi.org/10.1016/0370-2693(78)90566-X}{\textit{Phys. Lett.~B}} \textbf{74} (1978), 249--251.

\bibitem{GH}
Gibbons G.W., Hawking S., Gravitational multi-instantons, \href{https://doi.org/10.1016/0370-2693(78)90478-1}{\textit{Phys.
 Lett.~B}} \textbf{78} (1978), 430--432.

\bibitem{GP}
Gibbons G.W., Perry M.J., New gravitational instantons and their interactions,
 \href{https://doi.org/10.1103/PhysRevD.22.313}{\textit{Phys. Rev.~D}} \textbf{22} (1980), 313--321.

\bibitem{harmark}
Harmark T., Stationary and axisymmetric solutions of higher-dimensional general
 relativity, \href{https://doi.org/10.1103/PhysRevD.70.124002}{\textit{Phys. Rev.~D}} \textbf{70} (2004), 124002, 25~pages,
 \href{https://arxiv.org/abs/hep-th/0408141}{arXiv:hep-th/0408141}.

\bibitem{hatcher}
Hatcher A., Course notes: {B}asic 3-manifold topology,
 \url{https://pi.math.cornell.edu/~hatcher/}.

\bibitem{P}
Page D., Taub-{NUT} instantons with an horizon, \href{https://doi.org/10.1016/0370-2693(78)90016-3}{\textit{Phys. Lett.~B}}
 \textbf{78} (1978), 249--251.

\bibitem{saveliev}
Saveliev N., Lectures on the topology of 3-manifolds. {A}n introduction to the
 {C}asson invariant, \textit{{D}e Gruyter Textb.}, \href{https://doi.org/10.1515/9783110250367}{{D}e Gruyter}, Berlin, 2012.

\bibitem{T}
Tod P., One-sided type {D} {R}icci-flat metrics, \href{https://arxiv.org/abs/2003.03234}{arXiv:2003.03234}.

\end{thebibliography}
\end{document}